\documentclass{article}
\usepackage [utf8] {inputenc}
\usepackage{mathtools}
\usepackage{amsthm}
\usepackage{amssymb}
\usepackage{enumerate}
\usepackage{hyperref}

\newtheorem{theorem}{Theorem}[section]
\newtheorem{proposition}[theorem]{Proposition}
\newtheorem{lemma}[theorem]{Lemma}
\newtheorem{corollary}[theorem]{Corollary}

\DeclareMathOperator{\Sym}{Sym}
\DeclareMathOperator{\Aut}{Aut}
\DeclareMathOperator{\Hom}{Hom}
\DeclareMathOperator{\End}{End}
\DeclareMathOperator{\GL}{GL}
\newcommand{\X}{\mathfrak{X}}

\title{Polynomial-time isomorphism test for solvable groups with abelian Sylow subgroups}
\author{Saveliy V. Skresanov}
\date{}

\begin{document}
\maketitle

\begin{abstract}
	The group isomorphism problem in computational complexity asks whether two finite groups given
	by their Cayley tables are isomorphic or not. Although polynomial-time isomorphism tests exist
	for many specific types of groups, no general polynomial-time algorithm is known,
	classes of solvable and nilpotent groups being the main obstacles.
	In 2012 Babai and Qiao gave a polynomial-time isomorphism test for the class of solvable groups
	admitting normal series with abelian Sylow factors. We generalize their result and give a polynomial-time isomorphism
	test for solvable A-groups, i.e.\ solvable groups with abelian Sylow subgroups.
	The algorithm heavily relies both on the computational methods developed by Babai and Qiao, and structural properties of A-groups.
\end{abstract}

\section{Introduction}

The \emph{group isomorphism problem} in computational complexity is a problem of testing whether two finite groups given by their Cayley tables
are isomorphic or not. Although computations with groups given by Cayley tables are impractical for most orders, the Cayley table model
serves as a useful theoretical tool. On the one hand, many group-theoretical tasks like computing centralizers, normalizers, Sylow subgroups etc.\ are performed
trivially in polynomial time for groups given by Cayley tables. On the other hand, the group isomorphism problem is polynomial-time reducible
to the graph isomorphism problem and hence serves as a lower bound on the complexity of the latter.

It is well-known~\cite{miller} that isomorphism of groups of order \( n \) can be tested in time \( n^{\log n + O(1)} \),
since such a group has at most \( \log n \) generators (all logarithms are base~\( 2 \) in this work).
Since Babai's algorithm~\cite{babai} for testing isomorphism of graphs on \( n \) vertices in time \( n^{O((\log n)^c)} \), \( c > 1 \),
the group isomorphism problem became one of the main obstacles to further reducing the time complexity of graph isomorphism.

There has been a lot of progress in the group isomorphism problem for specific classes of groups, see the introduction of~\cite{grochow}
for an overview, and~\cite{brooksbank, ivanyos, grochowPgrs} for recent advancements in the class of \( p \)-groups.
We note here that, perhaps unsurprisingly, isomorphism of abelian groups can be tested in polynomial time and even in linear time~\cite{kavitha}.
On the other end of the spectrum, one can test in polynomial time isomorphism of groups with no nontrivial normal abelian subgroups or,
in other words, groups with trivial solvable radical~\cite{trivrad}. Partly because of the latter result, it is widely believed
that solvable groups or even nilpotent groups are the hardest cases of the group isomorphism problem, see~\cite[Subsection~1.1]{absyl} and \cite[Subsection~1.2]{grochow}.

Sylow subgroups of a finite group are nilpotent, and although there is no known reduction of the group isomorphism problem to the
isomorphism problem for Sylow subgroups of a group, it is natural to consider some group classes with restrictions on the Sylow structure.
One such class of groups was considered by Babai and Qiao in~\cite{absyl}. A finite group \( G \) has an \emph{abelian Sylow tower}
if there is a normal series \( G = N_1 \unrhd N_2 \unrhd \dots \unrhd N_k = 1 \), such that for every \( i = 1, \dots, k-1 \) the group
\( N_i/N_{i+1} \) is abelian and isomorphic to a Sylow subgroup of \( G \). A group with an abelian Sylow tower is solvable
and all its nilpotent subgroups are abelian.
The main result of~\cite{absyl} gives a polynomial-time isomorphism test for groups with abelian Sylow towers.

\begin{proposition}[{\cite[Theorem~1.1]{absyl}}]\label{atower}
	Let \( G \) and \( G_0 \) be groups given by their Cayley tables, and assume that \( G \) and \( G_0 \) have abelian Sylow towers.
	In polynomial time one can check if \( G \) and \( G_0 \) are isomorphic and find the coset of isomorphisms, if they are.
\end{proposition}

We note that the set of isomorphisms between \( G \) and \( G_0 \) is either empty, or a coset of \( \Aut(G) \) in the symmetric group \( \Sym(G \sqcup G_0) \),
where \( G \sqcup G_0 \) is the disjoint union of \( G \) and~\( G_0 \). We can compute this coset if we can compute permutation generators of \( \Aut(G) \)
and a coset representative.

Groups with abelian Sylow towers are a particular case of \emph{A-groups}, that is, groups with abelian Sylow subgroups.
The class of A-groups is closed with respect to taking subgroups, quotients and direct products,
and as in the case of groups with abelian Sylow towers, every nilpotent subgroup of an A-group is abelian.

The main result of this paper is a polynomial-time isomorphism test for solvable A-groups given by Cayley tables.
Note that one can easily check in polynomial time if a group belongs to this class.

\begin{theorem}\label{main}
	Let \( G \) and \( G_0 \) be solvable groups given by their Cayley tables, and assume that all Sylow subgroups of \( G \) and \( G_0 \) are abelian.
	In polynomial time one can check if \( G \) and \( G_0 \) are isomorphic and find the coset of isomorphisms, if they are.
\end{theorem}

It is mentioned in~\cite[Subsection~1.1]{absyl} that there are at most \( n^{O(\log n)} \) isomorphism classes of solvable A-groups of order at most \( n \),
and at least \( n^{\Omega(\log n)} \) of those are groups with abelian Sylow towers. Notice that a solvable A-group does not always have
an abelian Sylow tower, the group \( \Sym(3) \times \mathrm{Alt}(4) \) being an example,
so Theorem~\ref{main} is an improvement to Proposition~\ref{atower}.
We note that there exist nonsolvable A-groups, for instance, \( \mathrm{Alt}(5) \), and in Section~\ref{secFinal}
we discuss what is the main bottleneck towards generalizing Theorem~\ref{main} to the nonsolvable case.

Our proof of Theorem~\ref{main} follows the general strategy from~\cite{absyl}, so we explain the latter first.
If the group \( G \) has an abelian Sylow tower, then one can decompose \( G \) as a semidirect product \( G = A \rtimes H \),
where \( A \) is an abelian \( p \)-group, and \( |H| \) is not divisible by~\( p \). By~\cite[Theorem~1.2]{absyl},
the isomorphism problem for \( G \) reduces to the isomorphism problem for \( H \), if \( \Aut(H) \) is given to the algorithm as part of input.
This reduction allows us to work along the abelian Sylow tower, until we are left with the isomorphism problem for abelian groups.
The reduction itself is nontrivial, and relies on polynomial-time solutions to two problems about representations of \( H \) on \( A \):
the ``representation-transporting automorphisms problem''~\cite[Problem~2]{absyl} and the ``intertwining automorphisms problem''~\cite[Problem~3]{absyl}
(we postpone the formal treatment of these problems to Sections~\ref{secRtp} and~\ref{secInta}).
Polynomial-time algorithms for these tasks utilize representation theory of finite groups and heavily depend on the fact that \( |A| \) and \( |H| \) are coprime.

In order to prove Theorem~\ref{main}, we first show that the isomorphism problem for A-groups can be reduced to the problem of computing
the full group of automorphisms. The argument here follows~\cite{grpreduct}, but works inside the class of solvable A-groups.
Unfortunately the main result of~\cite{grpreduct} was formulated for the class of all finite groups, so in Section~\ref{secAut}
we reprove that result for an arbitrary class of groups closed with respect to direct products and direct factors.

Now, when \( G \) is a solvable A-group, we no longer have a decomposition \( G = A \rtimes H \),
where \( A \) is an abelian \( p \)-group and \( |H| \) is not divisible by~\( p \). Instead, as follows from the results of Broshi~\cite{broshi},
there exists a characteristic abelian \( p \)-subgroup \( A \) of \( G \) such that \( G = A \rtimes H \) for some subgroup \( H \) of \( G \),
and \( H \) is in a precise sense unique up to conjugation in~\( G \).
Since a Sylow \( p \)-subgroup of \( G \) is abelian, all \( p \)-elements of \( H \) act trivially on \( A \),
and hence the subgroup \( K \) of \( H \) generated by all \( p \)-elements lies in the kernel of the action of \( H \) on \( A \).
So the questions about representations of \( H \) on \( A \) reduce to the coprime case of representations of \( H/K \) on \( A \),
where the representation-transporting and intertwining automorphisms problems can be solved in polynomial time.
Hence to compute \( \Aut(G) \) it suffices to compute \( \Aut(H) \), so we proceed inductively until \( G \) is trivial.

We note that the results of Babai and Qiao~\cite{absyl}, which are instrumental to our work,
were published in a conference proceedings and their proofs were often sketched
or postponed to the full version. As far as the author is aware, the full version has never appeared in print, so in order to make
this paper as self-contained as possible, we provide complete proofs of all results from~\cite{absyl} that we require.
In many cases we either generalize the results to better suit our needs, or provide an alternative proof. A notable example
is Proposition~\ref{pbot} which resolves the intertwining automorphisms problem without any assumptions on coprimality or
the groups being A-groups.

The structure of the paper is as follows. In Section~\ref{secPre} we give group-theoretical and computational preliminaries,
Sections~\ref{secRtp} and~\ref{secInta} are devoted to the representation-transporting automorphisms problem and the intertwining automorphisms problem,
respectively. In Section~\ref{secMain} we prove the main result, and Section~\ref{secAut} gives polynomial-time reductions
between various natural problems on groups, including the reduction from the group isomorphism problem to the group automorphism problem.
In Section~\ref{secFinal} we give some remarks about the proof of Theorem~\ref{main} and discuss potential approaches for a polynomial-time
isomorphism test for arbitrary, not necessarily solvable, A-groups.

\section{Preliminaries}\label{secPre}

We consider only finite groups in this work. If \( G \) is a group and \( S \subseteq G \) then \( \langle S \rangle \) denotes
the subgroup of \( G \) generated by \( S \). Given groups \( G \) and \( H \), let \( \Hom(G, H) \) denote the set of all
homomorphisms from \( G \) to \( H \). Let \( \End(G) = \Hom(G, G) \) denote the monoid of all endomorphisms of \( G \), and let \( \Aut(G) \leq \End(G) \)
denote the full group of automorphisms of \( G \). Given a set \( \Omega \), the symmetric group on \( \Omega \) is denoted by \( \Sym(\Omega) \).

We compose maps from the left to right, so for \( f : G \to H \) and \( g : H \to K \) we have \( f \circ g : G \to K \).
We consider right group actions only unless stated otherwise, in particular, if \( \phi \in \Aut(G) \) and \( g \in G \) then we write
\( g^\phi = \phi(g) \) for the image of \( g \) under \( \phi \). Similarly, if \( \pi \in \Sym(\Omega) \) and \( x \in \Omega \),
then \( x^\pi = \pi(x) \) is the image of the point \( x \) under the permutation \( \pi \). If \( g, h \in G \) then \( g^h = h^{-1}gh \)
denotes the conjugation action of \( G \) on itself.

For a prime \( p \), we say that \( G \) is a \emph{\( p \)-group} if its order \( |G| \) is a power of~\( p \).
The group \( G \) is a \emph{\( p' \)-group}, if \( |G| \) is not divisible by~\( p \). A subgroup \( H \leq G \)
is a \emph{Hall subgroup}, if \( |H| \) is coprime with the index \( |G : H| \).
An abelian \( p \)-group is called \emph{homocyclic} if it is a direct product of cyclic groups of the same order.

Let \( H \) be a group and let \( A \) be an abelian \( p \)-group.
A \emph{representation} of \( H \) on \( A \) is a homomorphism \( \alpha : H \to \Aut(A) \).
Note that \( \alpha \) gives \( A \) the structure of a right \( H \)-module.
If a group \( G \) can be decomposed into a semidirect product \( G = A \rtimes H \),
then the conjugation action of \( H \) on \( A \) defines a representation.
In the other direction, if we are given a representation \( \alpha : H \to \Aut(A) \), then we can
construct a semidirect product \( A \rtimes H \) where the conjugation action of \( H \) on \( A \)
coincides with \( \alpha \).

The group \( \Aut(A) \) acts on the set of all representations of \( H \) on \( A \) by the following rule:
for \( \psi \in \Aut(A) \) and \( \alpha : H \to \Aut(A) \) define \( \alpha^\psi : H \to \Aut(A) \) by
\( \alpha^\psi(h) = \psi^{-1} \cdot \alpha(h) \cdot \psi \) for \( h \in H \). For \( \alpha, \beta : H \to \Aut(A) \)
we will write \( \alpha \sim \beta \) if there exists \( \psi \in \Aut(A) \) such that \( \alpha^\psi = \beta \).
Clearly \( \sim \) is an equivalence relation.

When \( A \) is an elementary abelian \( p \)-group, it can be viewed as a finite vector space over the prime field \( \mathbb{Z}/p\mathbb{Z} \),
so our notion of a representation coincides with the classical notion from representation theory.
In particular, in this case the relation \( \sim \) is the usual equivalence of representations.
We refer the reader to~\cite[Chapter~3]{gorenstein} for basic facts about representations.
If \( p \) does not divide \( |H| \), then by Maschke's theorem~\cite[Theorem~3.3.1]{gorenstein} every representation of \( H \) on \( A \) can be decomposed
into a direct sum of irreducible representations. Furthermore, every representation is determined up to equivalence by
its irreducible constituents and their multiplicities, and there are at most \( |H| \) equivalence types of
irreducible representations of \( H \), see~\cite[Theorem~3.6.14]{gorenstein}.

Note that \( \Aut(H) \) also acts on the set of all representations of \( H \) on \( A \).
Indeed, for \( \phi \in \Aut(H) \) and \( \alpha : H \to \Aut(A) \) define \( \alpha^\phi : H \to \Aut(A) \)
by \( \alpha^\phi(h) = \alpha(h^{\phi^{-1}}) \) for \( h \in H \).
This action commutes with the action of \( \Aut(A) \) on representations, so there is an action of \( \Aut(A) \times \Aut(H) \)
on the set of representations of \( H \) on \( A \).

We are interested in algorithms for finite groups and related structures.
We say that an algorithm works in \emph{polynomial time} if it works in time which is bounded by a polynomial
in the length of the input. A group \( G \) is given by its \emph{Cayley table}, if we input the group to the algorithm
as a \( |G| \times |G| \) multiplication table. In this case the length of the input is \( O(|G|^2 \log |G|) \),
so the algorithm works in polynomial time if it works in time polynomial in \( |G| \).
A subset of a group given by a Cayley table is specified by a list of its elements.
A homomorphism \( f \) from \( G \) to some other group \( H \) is given by the list of images of elements of \( G \),
i.e.\ by \( (f(g))_{g \in G} \).

Note that if \( G \) is given by its Cayley table, then the following tasks can be performed in polynomial time:
for \( S \subseteq G \) we can compute \( \langle S \rangle \); we can find \( |G| \) and its decomposition into prime factors;
if \( H \) is a group given by its Cayley table, then we can compute the Cayley table of \( G \times H \) and corresponding embeddings
of \( G \) and \( H \) into this direct product; we can decompose \( G \) into a direct product of directly-indecomposable subgroups~\cite{kayal};
given a normal subgroup \( N \) of \( G \) we can compute \( G/N \) and a natural homomorphism from \( G \) onto \( N \);
we can compute centralizers and normalizers of subsets in \( G \); the derived subgroup of \( G \) can be found.

If \( G \leq \Sym(\Omega) \), we say that \( G \) is given by \emph{permutation generators}
if we are given a subset \( S \subseteq \Sym(\Omega) \) which generates \( G \). If \( n = |\Omega| \),
then the length of input in this case is \( O(|S|n \log n) \) and an algorithm works in polynomial time
if it works in time polynomial in \( |S| \) and \( n \). It is well known, see~\cite[Exercise~4.1]{seress},
that any generating set \( S \) can be reduced in time polynomial in \( |S| \) and \( n \) to a generating set of size at most \( n^2 \).
Hence we can assume that an algorithm for permutation groups works in polynomial time if it works in time polynomial in~\( n \).
Subgroups of \( G \) are also specified via their generating sets. A homomorphism \( f \) from \( G \) to some other group \( H \)
is given by the list of images of generators of~\( G \), i.e.\ by \( (f(g))_{g \in S} \).
A coset \( Hx \subseteq G \) is given by generators of the subgroup \( H \leq G \) and a representative \( y \in Hx \).

We refer the reader to~\cite[Subsection~3.1]{seress} for a list of polynomial-time algorithms for permutation groups.
In particular, the following tasks can be performed in polynomial time for a group \( G \leq \Sym(\Omega) \) given by permutation generators:
compute the solvable radical; find the orbits of \( G \); find the pointwise stabilizer for any \( \Delta \subseteq \Omega \);
find the kernel and image of any homomorphism \( f : G \to \Sym(\Pi) \) in time polynomial in \( |\Omega| \) and \( |\Pi| \).
Since a group given by its Cayley table admits a regular permutation representation, we can compute the solvable radical
of such a group in polynomial time as well. Note that if a group \( H \) is given by its Cayley table, then any subgroup \( G \leq \Aut(H) \)
can be viewed as a permutation group on \( |H| \), so the above polynomial-time algorithms apply if the generators of \( G \) are known.

Let \( A \) be a finite abelian group. If \( A \) is isomorphic to a direct product of cyclic groups
of orders \( n_1, \dots, n_k \) with generators \( g_1, \dots, g_k \in A \), then we call \( g_1, \dots, g_k \) the \emph{cyclic generators} of \( A \).
We say that \( A \) is given to an algorithm by its cyclic generators, if the algorithm is given a \( k \)-tuple \( (n_1, \dots, n_k) \) as input;
the length of the input is \( O(\sum_{i=1}^k \log n_i) = O(\log |A|) \). Subgroups of \( A \) are specified using their generating sets.
A homomorphism \( f \) from \( A \) to some other group \( H \) is given by the list \( (f(g_i))_{i=1,\dots,k} \) of images of cyclic generators.
We refer the reader to~\cite[Chapter~2]{iuliana-thesis} for some basic polynomial-time algorithms for abelian groups given by cyclic generators.
Note that if \( A \) is an abelian group given by its Cayley table, then we can find some cyclic generators of \( A \) in polynomial time.

A finite unitary ring \( R \) is given by its cyclic generators if we are given cyclic generators \( g_1, \dots, g_k \) of the underlying
additive group and are also given \( k^3 \) integers \( \alpha_{ijs} \) such that \( g_i \cdot g_j = \sum_{s=1}^k \alpha_{ijs} g_s \);
the length of the input in this case can be bounded by \( O(k^2 \log |R|) \). Let \( R^\times \) denote the unit group of \( R \). We say that \( R^\times \)
is given by its multiplicative generators if we are given a subset of elements of \( R \) which generates \( R^\times \) multiplicatively, as a group.
Note that if \( A \) is a finite abelian group given by its Cayley table, then we can compute cyclic generators of the full endomorphism ring \( \End(A) \)
in polynomial time. Given a unitary subring \( R \leq \End(A) \), the unit group \( R^\times \) can be viewed as a permutation group on \( A \).
Finally, if \( A \) is an \( R \)-module, then we say that it is given by cyclic generators if we are given cyclic generators of the abelian group \( A \),
cyclic generators of the ring \( R \), and are given the action of \( R \) on \( A \) as a homomorphism
\( f : R \to \End(A) \). The length of input in this case is bounded polynomially in terms of \( \log |A| \) and \( \log |R| \).

\section{Representation-transporting automorphisms}\label{secRtp}

Let \( H \) and \( A \) be finite groups given by their Cayley tables, and let \( A \) be an abelian \( p \)-group
for \( p \) not dividing \( |H| \). Suppose we are given representations \( \alpha, \beta : H \to \Aut(A) \)
and the generators of \( P \leq \Aut(H) \) viewed as a permutation group on \( H \). In this section we are interested in computing the following set:
\[ P_{\alpha \to \beta} = \{ \phi \in P \mid \alpha^\phi = \beta \}. \]
Note that \( P_{\alpha \to \beta} \) is either empty, or is a coset of \( P_{\alpha \to \alpha} \) in \( P \), so it can be specified
by the generators of \( P_{\alpha \to \alpha} \) and a representative.

In the case when \( P = \Aut(H) \), the problem of computing \( P_{\alpha \to \beta} \) was called the
``representation-transporting automorphisms problem'' in~\cite[Problem~2]{absyl}.
Our situation is a bit more general as \( P \) is an arbitrary subgroup of \( \Aut(H) \), but the overall idea
of the algorithm follows~\cite[Sections~5 and~6]{absyl}. First we resolve the case when \( A \) is elementary abelian (cf.~\cite[Section~5]{absyl}),
and then we follow the proof outline in~\cite[Section~6]{absyl} to reduce the general case to the elementary abelian one.

\subsection{Elementary abelian case}

Let \( P \leq \Sym(\Omega) \) be a permutation group. A function \( f : \Omega \to \{ 1, \dots, l \} \), \( l \geq 1 \),
is called a \emph{string} on the set \( \Omega \) with the alphabet \( \{ 1, \dots, l \} \). The action of \( \pi \in P \)
on the set of all strings is defined by \( f^\pi(x) = f(x^{\pi^{-1}}) \), \( x \in \Omega \).
Two strings \( f_1, f_2 \) are \emph{\( P \)-isomorphic} if there exists \( \pi \in P \) such that \( f_1^\pi = f_2 \). Let
\[ \mathrm{Iso}_P(f_1, f_2) = \{ \pi \in P \mid f_1^{\pi} = f_2 \} \]
denote the set of all \( P \)-isomorphisms between \( f_1 \) and \( f_2 \). When nonempty, this set is a coset of \( \mathrm{Iso}_P(f_1, f_1) \) inside \( P \).

For \( A, B \subseteq \Omega \) define \( P_{A \to B} = \{ \pi \in P \mid A^\pi = B \} \). Again, this set is either empty
or a coset of \( P_{A \to A} \).

The following result is proved in~\cite[Subsection~5.1.1]{absyl}, but we provide the argument here for completeness.

\begin{lemma}[{\cite[Proposition~4]{absyl}}]\label{lsubset}
	Let \( P \leq \Sym(\Omega) \) be given by permutation generators. Given \( A, B \subseteq \Omega \) with \( |A| = |B| \), we can compute
	\[ P_{A \to B} = \{ \pi \in P \mid A^\pi = B \} \]
	in time polynomial in \( 2^{|A|} \) and \( |\Omega| \).
\end{lemma}
\begin{proof}
	Set \( k = |A| = |B| \). Order elements of \( A \) arbitrarily as \( A = \{ x_1, \dots, x_k \} \)
	and let \( A_i = \{ x_1, \dots, x_i \} \), \( i = 1, \dots, k \). For every \( i = 1, \dots, k \)
	and \( C \subseteq B \) with \( |C| = i \) we will compute \( P_{A_i \to C} \) and put it into a dynamic programming table.
	For \( i = 1 \) we need to compute \( P_{\{ x_1 \} \to \{ b \}} \) for \( b \in B \), which can be done
	in polynomial time by the standard orbit-stabilizer algorithm for permutation groups.
	Suppose the table is filled for \( j = 1, \dots, i-1 \). Then note that
	\[ P_{A_i \to C} = \bigcup_{b \in C} (P_{A_{i-1} \to C \setminus \{ b \}})_{\{ x_i \} \to \{ b \}}. \]
	The set \( P_{A_{i-1} \to C \setminus \{ b \} } \) can be looked up in the table, and \( (P_{A_{i-1} \to C \setminus \{ b \}})_{\{ x_i \} \to \{ b \}} \)
	can be computed in polynomial time as in the case \( i = 1 \) by the orbit-stabilizer algorithm. Finally, we need to take a union of at most \( i \leq k \) cosets
	\( P_1\pi_1, \dots, P_s\pi_s \). Since \( P_{A_i \to C} \) is a coset, we can take \( \pi_1 \) as a representative
	and \( P_{A_i \to A_i} = \langle P_1, P_2\pi_2\pi_1^{-1}, \dots, P_s\pi_s\pi_1^{-1} \rangle \).

	There are at most \( 2^k \) table entries, and we can compute each entry in time polynomial in \( |\Omega| \),
	so the algorithm works in time polynomial in \( 2^k \) and \( |\Omega| \).
\end{proof}

\begin{corollary}[{\cite[Corollary~5.1]{absyl}}]\label{cstrings}
	Let \( P \leq \Sym(\Omega) \) be given by permutation generators, and let \( f_1, f_2 : \Omega \to \{ 1, \dots, l \} \), \( l \geq 2 \), be two strings.
	If for \( j = 1, 2 \) and some \( k \) we have \( \sum_{i=1}^{l-1} |f_j^{-1}(i)| \leq k \), then \( \mathrm{Iso}_P(f_1, f_2) \)
	can be computed in time polynomial in \( 2^k \) and \( |\Omega| \).
\end{corollary}
\begin{proof}
	Define \( P_0 = P \) and \( P_i = (P_{i-1})_{f_1^{-1}(i) \to f_2^{-1}(i)} \) for \( i = 1, \dots, l-1 \).
	We can compute \( P_i \), \( i = 1, \dots, l-1 \), in time polynomial in \( 2^k \) and \( |\Omega| \) by Lemma~\ref{lsubset}.
	Now it is left to note that \( \mathrm{Iso}_P(f_1, f_2) = P_{l-1} \).
\end{proof}

The following lemma establishes some basic polynomial-time algorithms for representations of finite groups.
We note that these tasks can be performed in polynomial time in a much wider setting, see, for instance,~\cite[Chapter~7]{holt},
but in our case elementary methods suffice.

\begin{lemma}\label{ldecomp}
	Let \( H \) and \( A \) be finite groups given by their Cayley tables, and assume that \( A \)
	is an elementary abelian \( p \)-group for \( p \) not dividing \( |H| \).
	\begin{enumerate}
		\item Given irreducible representations \( \alpha, \beta : H \to \Aut(A) \), we can
			test in polynomial time whether \( \alpha \sim \beta \) or not.
		\item If we are given a (not necessarily irreducible) representation \( \alpha : H \to \Aut(A) \),
			then there are at most \( |A| \) irreducible submodules of \( A \) and they can be enumerated in polynomial time.
			In particular, we can decompose \( \alpha \) into a direct sum of its irreducible constituents in polynomial time.
	\end{enumerate}
\end{lemma}
\begin{proof}
	We prove part~1 first. Note that \( \alpha \) gives \( A \) a structure of a right \( H \)-module.
	Since this module is irreducible, it is cyclic, i.e.\ it can be generated by an \( \alpha(H) \)-orbit of one vector.
	Same reasoning applies to \( \beta \), and \( \alpha \sim \beta \) if and only if the corresponding modules are isomorphic.
	We can check the latter by bruteforce: for every vectors \( u, v \in A \) such that orbits \( u^{\alpha(H)} \) and \( v^{\alpha(H)} \) generate \( A \)
	as a group, check if the map \( \phi : A \to A \) with \( \phi(u) = v \) can be extended to an isomorphism of \( H \)-modules. There are at most \( |A|^2 \)
	possibilities for \( u \) and \( v \), and for each choice the check can be performed in polynomial time.

	To prove part~2 we again view \( A \) as an \( H \)-module. Every irreducible submodule is generated by an orbit of a single vector,
	so there are at most \( |A| \) irreducible submodules. For each submodule generated by an orbit of a single vector we can check in polynomial
	time whether it is irreducible or not; it suffices to check that it does not contain a proper nontrivial submodule, which can again be done by bruteforce.
	Therefore we can enumerate all irreducible \( H \)-submodules of \( A \).

	We can decompose \( A \) into a direct sum of irreducibles by a greedy algorithm, which will compute a growing series
	of submodules \( A_i \), \( i = 1, \dots, k \), where \( k \) is the number of irreducibles summands in the decomposition, and \( A_k = A \).
	Set \( A_1 \) to be any irreducible submodule of \( A \). If \( A_i < A \), then find an irreducible submodule \( M \) of \( A \)
	such that \( M \) does not lie in \( A_i \); we define \( A_{i+1} = A_i \oplus M \). It is clear that this algorithm works correctly
	and will find a decomposition of \( A \) into irreducible components in polynomial time. Projection of \( \alpha \) onto these components
	gives a decomposition of \( \alpha \) into irreducible constituents.
\end{proof}

We are ready to solve the representation-transporting automorphisms problem in the elementary abelian case.

\begin{proposition}[cf. {\cite[Theorem~5.3]{absyl}}]\label{elabcase}
	Let \( H \) and \( A \) be finite groups given by their Cayley tables, and assume that \( A \)
	is an elementary abelian \( p \)-group for \( p \) not dividing \( |H| \).
	Suppose \( P \leq \Aut(H) \) is given by permutation generators, and we are given representations \( \alpha, \beta : H \to \Aut(A) \).
	Then \( P_{\alpha \to \beta} = \{ \phi \in P \mid \alpha^\phi \sim \beta \} \) can be computed in polynomial time.
\end{proposition}
\begin{proof}
	By Lemma~\ref{ldecomp}~(2), we can decompose \( \alpha \) and \( \beta \) into sums of irreducible representations in polynomial time.
	Let \( \Omega \) be the union of orbits of these representations under the action of \( P \).
	Since \( H \) has at most \( |H| \) irreducible representations over \( \mathbb{Z}/p\mathbb{Z} \) and we can
	test if two given irreducible representations are equivalent by Lemma~\ref{ldecomp}~(1), the set \( \Omega \)
	can be computed in polynomial time. We have \( |\Omega| \leq |H| \).

	Let \( m \) be the rank of \( A \), so \( |A| = p^m \).
	Let \( \alpha' : \Omega \to \{ 0, 1, \dots, m \} \) be the string of multiplicities of \( \alpha \),
	that is, \( \alpha'(\phi) \) is the multiplicity of the representation \( \phi \in \Omega \) in \( \alpha \).
	Clearly \( \alpha' \) can be computed in polynomial time,
	and we can similarly define and compute \( \beta' : \Omega \to \{ 0, 1, \dots, m \} \). The group \( P \) acts on \( \Omega \)
	by permuting the irreducible representations, and since two representations are equivalent
	if and only if they have the same decompositions into irreducible constituents we have
	\( P_{\alpha \to \beta} = \mathrm{Iso}_P(\alpha', \beta') \).
	Suppose that \( \alpha \) decomposes into irreducible constituents \( \alpha_i \), \( i = 1, \dots, k \),
	with multiplicities \( n_i \) and degrees \( d_i \).
	Then \( \sum_{i = 1}^k n_i d_i = m \), in particular, \( \sum_{j=1}^m |(\alpha')^{-1}(j)| \leq m \).
	Similarly, \( \sum_{j=1}^m |(\beta')^{-1}(j)| \leq m \). By Corollary~\ref{cstrings}, we can compute
	\( \mathrm{Iso}_P(\alpha', \beta') \) in time polynomial in the degree of \( P \) and \( 2^m \).
	Since \( 2^m \leq |A| \) and \( A \) is given by its Cayley table, our algorithm works in polynomial time.
\end{proof}

\subsection{General case}

If \( A \) is a finite abelian \( p \)-group written additively, then let \( pA = \{ p \cdot x \mid x \in A \} \)
denote the subgroup of \( A \) consisting of all \( p \)-th powers of elements of~\( A \). Note that it coincides with the Frattini subgroup of \( A \).

\begin{lemma}\label{twohomo}
	Let \( A_1 \) and \( A_2 \) be finite abelian homocyclic \( p \)-groups of different exponents.
	If \( u_{12} \in \Hom(A_1, A_2) \) and \( u_{21} \in \Hom(A_2, A_1) \), then \( u_{12}u_{21} \in p \cdot \Hom(A_1, A_1) \).
\end{lemma}
\begin{proof}
	Let \( e_i \) be the exponent of \( A_i \), \( i = 1, 2 \).
	Suppose that \( e_1 > e_2 \). Then \( \Hom(A_2, A_1) = \Hom(A_2, pA_1) \), since \( pA_1 \) is the subgroup of elements of order strictly smaller than \( e_1 \).
	Let \( a_1, \dots, a_k \) be cyclic generators of \( A_1 \).
	Then \( (u_{12}u_{21})(a_i) = u_{21}(u_{12}(a_i)) \in pA_1 \),
	hence there exist elements \( b_i \in A_1 \) such that \( (u_{12}u_{21})(a_i) = p\cdot b_i \) for \( i = 1, \dots, k \).
	There exists \( \phi \in \Hom(A_1, A_1) \) defined by \( \phi(a_i) = b_i \), \( i = 1, \dots, k \).
	It follows that \( u_{12}u_{21} = p \cdot \phi \in p \cdot \Hom(A_1, A_1) \).

	Suppose that \( e_1 < e_2 \). Then \( \Hom(A_1, A_2) = \Hom(A_1, pA_2) \).
	Since \( u_{21}(pA_2) \leq pA_1 \), we have \( (u_{12}u_{21})(A_1) \leq pA_1 \)
	and by reasoning as in the previous case we derive \( u_{12}u_{21} \in p \cdot \Hom(A_1, A_1) \).
\end{proof}

For \( m \geq 1 \) and a commutative ring \( R \) let \( M_m(R) \) denote the ring of \( m \times m \) matrices over \( R \).

\begin{lemma}\label{homo}
	Let \( A \) be a finite homocyclic \( p \)-group of exponent \( e \) and rank~\( m \).
	Then \( \End(A) \simeq M_m(\mathbb{Z}/e\mathbb{Z}) \) and \( \Aut(A) \simeq \GL_m(\mathbb{Z}/e\mathbb{Z}) \).
	Moreover, let \( \overline{\phantom{a}} : \End(A) \to \End(A/pA) \) be a natural map sending \( \phi \in \End(A) \)
	to \( \overline{\phi} \in \End(A/pA) \) defined by \( \overline{\phi}(x+pA) = \phi(x) + pA \), \( x \in A \).
	Then this map is a surjective ring homomorphism and its restriction \( \overline{\phantom{a}} : \Aut(A) \to \Aut(A/pA) \) is
	a surjective group homomorphism.
\end{lemma}
\begin{proof}
	Let \( a_1, \dots, a_m \) be cyclic generators of \( A \). An endomorphism of \( A \) is completely defined by its
	action on \( a_1, \dots, a_m \). Since every element of \( A \) has order dividing \( e \), for every \( b_1, \dots, b_m \in A \)
	there exists \( \phi \in \End(A) \) such that \( \phi(a_i) = b_i \), \( i = 1, \dots, m \).
	If \( b_i = u_{i1}a_1 + \dots + u_{im}a_m \), \( i = 1, \dots, m \), \( u_{ij} \in \mathbb{Z}/e\mathbb{Z} \),
	then \( \phi \) can be identified with the matrix \( (u_{ij}) \).
	This identification establishes the isomorphism \( \End(A) \simeq M_m(\mathbb{Z}/e\mathbb{Z}) \).

	The group of invertible elements of \( M_m(\mathbb{Z}/e\mathbb{Z}) \) is precisely \( \GL_m(\mathbb{Z}/e\mathbb{Z}) \),
	hence \( \Aut(A) \simeq \GL_m(\mathbb{Z}/e\mathbb{Z}) \). Note that the map \( \overline{\phantom{a}} : \End(A) \to \End(A/pA) \)
	is a well-defined ring homomorphism since \( pA \) is a characteristic subgroup of \( A \).
	If \( \phi \in \End(A) \) is identified with a matrix from \( M_m(\mathbb{Z}/e\mathbb{Z}) \)
	and \( \overline{\phi} \in \End(A/pA) \) is identified with a matrix from \( M_m(\mathbb{Z}/p\mathbb{Z}) \),
	then entities of \( \overline{\phi} \) are equal to corresponding entities of \( \phi \) modulo~\( p \).
	It is clear that this map is surjective.

	Finally, a matrix is invertible if and only if its determinant is invertible.
	Since \( e \) is a power of a prime \( p \), an element from \( \mathbb{Z}/e\mathbb{Z} \)
	is invertible if and only if the corresponding integer is not divisible by~\( p \).
	It is now easy to see that for every matrix from \( \GL_m(\mathbb{Z}/p\mathbb{Z}) \)
	there is a preimage in \( \GL_m(\mathbb{Z}/e\mathbb{Z}) \) with respect to reduction modulo~\( p \),
	so the restriction \( \overline{\phantom{a}} : \Aut(A) \to \Aut(A/pA) \) is surjective.
\end{proof}

The following proposition collects some information on endomorphisms and automorphisms
of finite abelian \( p \)-groups. These facts are well-known to specialists, see the works~\cite{ranum, hillar, mader} and~\cite[Lemma~6.1]{absyl},
but, as far as the author is aware, have not been collected in one place in the manner suitable for our algorithmic applications.
We note that the claim about the Jacobson radical in part~(3) is not used in our further arguments and we provided it for completeness.

\begin{proposition}\label{autab}
	Let \( A \) be a finite abelian \( p \)-group written additively, and let \( A = A_1 \oplus \dots \oplus A_k \)
	be some decomposition into homocyclic components of different exponents.
	\begin{enumerate}
		\item Every endomorphism of \( A \) can be identified with a \( k \times k \) matrix
			\[ U =
			\begin{pmatrix}
				u_{11} & u_{12} & \dots & u_{1k}\\
				\vdots & \vdots & \dots & \vdots\\
				u_{k1} & u_{k2} & \dots & u_{kk}
			\end{pmatrix}
			\]
			where \( u_{ij} \in \Hom(A_i, A_j) \).
			The action of \( U \) on \( (x_1, \dots, x_k) \in A \), \( x_i \in A_i \), \( i = 1, \dots, k \),
			is given by
			\[ 
			xU = (u_{11}(x_1) + u_{21}(x_2) \dots + u_{k1}(x_k), \dots, u_{1k}(x_1) + u_{2k}(x_2) \dots + u_{kk}(x_k)).
			\]
			This action can be interpreted as a multiplication by a matrix, and the composition of endomorphisms
			corresponds to the matrix product.
		\item The endomorphism \( U = (u_{ij}) \) is an automorphism if and only if \( u_{ii} \in \Aut(A_i) \) for all \( i = 1, \dots, k \).
		\item There is a natural map
			\[ \Lambda : \End(A) \to \End(A_1 / pA_1) \oplus \dots \oplus \End(A_k / pA_k) \]
			which sends \( U = (u_{ij}) \in \End(A) \) to \( \Lambda(U) = (\overline{u_{11}}, \dots, \overline{u_{kk}}) \),
			where \( \overline{u_{ii}} \in \End(A_i / pA_i) \) is defined by \( \overline{u_{ii}}(x + pA_i) = u_{ii}(x) + pA_i \),
			\( i = 1, \dots, k \).

			This map is a surjective ring homomorphism, and its kernel \( J \) is equal to the set of matrices of the following form
			\[ \begin{pmatrix}
				p \cdot u_{11} & u_{12} & \dots & u_{1k}\\
				\vdots & \vdots & \dots & \vdots\\
				u_{k1} & u_{k2} & \dots & p \cdot u_{kk}
			   \end{pmatrix}
			\]
			where \( u_{ij} \in \Hom(A_i, A_j) \). The Jacobson radical of \( \End(A) \) is equal to \( J \).
		\item The restriction of \( \Lambda \) to \( \Aut(A) \) induces a surjective homomorphism of groups
			\[ \Lambda : \Aut(A) \to \Aut(A_1 / pA_1) \times \dots \times \Aut(A_k / pA_k) \]
			and the kernel of this group homomorphism is equal to \( K = 1 + J \), where \( 1 \) denotes the identity automorphism of \( A \).
			The largest normal \( p \)-subgroup of \( \Aut(A) \) is equal to~\( K \).
	\end{enumerate}
\end{proposition}
\begin{proof}
	Parts~1 and~2 are just a restatement of~\cite[Theorem~3.2~(1) and~(2)]{mader}; we note that~\cite{mader} uses column-vectors
	and includes trivial homocyclic summands, while we exclude trivial summands and adapt notation to use row-vectors.

	To prove part~3, first note that since \( pA_i \) is a characteristic subgroup of \( A_i \), the maps \( \overline{u_{ii}} \)
	and \( \Lambda \) are well-defined. We need to check that \( \Lambda \) is a ring homomorphism.
	The fact that \( \Lambda \) preserves addition is clear. Now let \( U = (u_{ij}) \in \End(A) \)
	and \( V = (v_{ij}) \in \End(A) \). The \( i \)-th diagonal element of \( UV \) is equal to
	\[ (UV)_{ii} = \sum_{j=1}^k u_{ij}v_{ji} = u_{ii}v_{ii} + \sum_{j\neq i} u_{ij}v_{ji}. \]
	By Lemma~\ref{twohomo} applied to \( u_{ij}v_{ji} \), \( j \neq i \) we derive
	\[ (UV)_{ii} \in u_{ii}v_{ii} + p \cdot \Hom(A_i, A_i), \]
	hence \( \overline{(UV)_{ii}} = \overline{u_{ii}v_{ii}} = \overline{u_{ii}} \cdot \overline{v_{ii}} \), \( i = 1, \dots, k \).
	Therefore \( \Lambda(UV) = \Lambda(U)\Lambda(V) \) and \( \Lambda \) is a ring homomorphism.
	\( \Lambda \) is surjective since by Lemma~\ref{homo} the maps \( \overline{\phantom{u}} : \End(A_i) \to \End(A_i/pA_i) \), \( i = 1, \dots, k \),
	are surjective.

	We want to compute the kernel \( J \) of \( \Lambda \).
	The kernel \( J \) consists of those matrices whose diagonal elements \( v_{ii} \) satisfy \( \overline{v_{ii}} = 0 \).
	This is equivalent to \( v_{ii}(x) \in pA_i \) for all \( x \in A_i \), which is equivalent to \( v_{ii} = p \cdot u_{ii} \)
	for some \( u_{ii} \in \Hom(A_i, A_i) \). Hence \( J \) consists exactly of those matrices which have homomorphisms
	of the form \( p \cdot u_{ii} \), \( i = 1, \dots, k \) on the diagonal.

	Since \( \Lambda \) maps onto a direct sum of matrix rings over fields, the Jacobson radical of \( \End(A) \) lies in \( J \).
	On the other hand, if \( u_{ii} \in \End(A_i) \), then \( p \cdot u_{ii} \) is nilpotent and hence \( 1 + p \cdot u_{ii} \) is invertible
	for \( i = 1, \dots, k \). Hence by part~2, every element of the form \( 1 + x \), \( x \in J \), lies in \( \Aut(A) \),
	so \( 1 + x \) is invertible and hence \( x \) is quasiregular.
	Thus \( J \) consists of quasiregular elements and therefore it is equal to the Jacobson radical of \( \End(A) \), see~\cite[Theorem~1.2.3]{herstein}.

	It is left to prove part~4.
	The fact that \( \Lambda : \Aut(A) \to \Aut(A_1 / pA_1) \times \dots \times \Aut(A_k / pA_k) \)
	is a surjective group homomorphism follows from part~(2) and Lemma~\ref{homo}.
	An element \( 1 + U \), \( U \in \End(A) \), lies in the kernel of \( \Lambda \) if and only if \( \Lambda(1+U) = 1 \).
	This is equivalent to \( \Lambda(U) = 0 \), so \( U \in J \). Hence the kernel \( K \leq \Aut(A) \) of the group homomorphism \( \Lambda \)
	is equal to \( 1 + J \). Since \( J \) is an abelian subgroup of \( \End(A) \), its order is a power of \( p \), so \( K \) is a \( p \)-group.
	Now, \( \Aut(A_1 / pA_1) \times \dots \times \Aut(A_k / pA_k) \) is a direct product of matrix rings over \( \mathbb{Z}/p\mathbb{Z} \),
	and it contains no nontrivial normal \( p \)-subgroups. Hence \( K \) is the largest normal \( p \)-subgroup of \( \Aut(A) \).
\end{proof}

The next result gives a criterion for equivalence of representations of a \( p' \)-group on an abelian \( p \)-group
in terms of equivalences of representations on elementary abelian \( p \)-groups.

\begin{lemma}\label{elabred}
	Let \( A \) be a finite abelian \( p \)-group written additively, and let \( A = A_1 \oplus \dots \oplus A_k \)
	be some decomposition into homocyclic components of different exponents.
	Let \( H \) be a finite \( p' \)-group.
	In the notation of Proposition~\ref{autab}, let \( \Lambda \) be the homomorphism from \( \Aut(A) \)
	onto \( \Aut(A_1 / pA_1) \times \dots \times \Aut(A_k / pA_k) \), and let \( \Lambda_i : \Aut(A) \to \Aut(A_i/pA_i) \)
	be the composition of \( \Lambda \) with the projection onto \( \Aut(A_i/pA_i) \) for \( i = 1, \dots, k \).
	Then for representations \( \alpha, \beta : H \to \Aut(A) \) we have \( \alpha \sim \beta \) if and only if
	\( \alpha \circ \Lambda_i \sim \beta \circ \Lambda_i \) for all \( i = 1, \dots, k \).
\end{lemma}
\begin{proof}
	If \( \alpha \sim \beta \), then there exists \( \psi \in \Aut(A) \) such that \( \alpha^\psi = \beta \).
	By applying \( \Lambda_i \) we yield \( (\alpha \circ \Lambda_i)^{\Lambda_i(\psi)} = \beta \circ \Lambda_i \)
	for all \( i = 1, \dots, k \). By definition this means that \( \alpha \circ \Lambda_i \sim \beta \circ \Lambda_i \) for all \( i = 1, \dots, k \).

	Now assume that \( \alpha \circ \Lambda_i \sim \beta \circ \Lambda_i \) for all \( i = 1, \dots, k \).
	There exist elements \( \psi_i \in \Aut(A_i/pA_i) \) such that \( (\alpha \circ \Lambda_i)^{\psi_i} = \beta \circ \Lambda_i \)
	for \( i = 1, \dots, k \). Since \( \Lambda \) is surjective, there exists an element \( \psi \in \Aut(A) \)
	such that \( \Lambda_i(\psi) = \psi_i \), \( i = 1, \dots, k \). We have \( (\alpha \circ \Lambda)^{\Lambda(\psi)} = \beta \circ \Lambda \).
	By replacing \( \alpha \) by \( \alpha^\psi \), without loss of generality we may assume that \( \alpha \circ \Lambda = \beta \circ \Lambda \).

	Let \( K \) be the kernel of \( \Lambda \).
	By Proposition~\ref{autab}~(4), \( K \) is a \( p \)-group. We have \( \alpha(H)K = \beta(H)K \), and \( \alpha(H), \beta(H) \) are \( p' \)-groups,
	hence by the Schur--Zassenhaus theorem there exists an element \( k \in K \) such that \( \alpha(H)^k = \beta(H) \). By replacing \( \alpha \) by \( \alpha^k \),
	without loss of generality we may assume that \( \alpha(H) = \beta(H) \). Since \( k \) lies in the kernel of \( \Lambda \),
	we still have \( \alpha \circ \Lambda = \beta \circ \Lambda \). So for all \( h \in H \) we have \( \alpha(h)K = \beta(h)K \) and therefore
	\[ \alpha(h) \in (\beta(h)K) \cap \alpha(H) = (\beta(h)K) \cap \beta(H) = \beta(h)(K \cap \beta(H)) = \{ \beta(h) \}, \]
	so \( \alpha(h) = \beta(h) \). Thus \( \alpha \sim \beta \), as required.
\end{proof}

Now we are ready to solve the representation-transporting automorphisms problem for arbitrary abelian \( p \)-groups.

\begin{proposition}\label{ptop}
	Let \( H \) and \( A \) be groups given by their Cayley tables, and assume we are given \( P \leq \Aut(H) \) 
	by permutation generators. Suppose that \( A \) is an abelian \( p \)-group, and \( H \) is a \( p' \)-group.
	If we are given representations \( \alpha, \beta : H \to \Aut(A) \) of \( H \) on \( A \), then one can compute
	\( P_{\alpha \to \beta} = \{ \phi \in P \mid \alpha^\phi \sim \beta \} \) in polynomial time.
\end{proposition}
\begin{proof}
	We can decompose \( A \) into a direct sum of its homocyclic components \( A = A_1 \oplus \dots \oplus A_k \)
	in polynomial time. Let us use the notation of Proposition~\ref{autab} and Lemma~\ref{elabred}.
	By Proposition~\ref{autab}~(1), every automorphism \( \psi \) of \( A \) can be identified with a \( k \times k \) matrix
	of homomorphisms. We can compute entities of that matrix in polynomial time by considering cyclic generators of \( A_i \), \( i = 1, \dots, k \),
	and their images with respect to \( \psi \). In particular, the map \( \Lambda \) and hence maps \( \Lambda_i \), \( i = 1, \dots, k \),
	are computable in polynomial time.

	By Lemma~\ref{elabred}, for \( \phi \in \Aut(H) \) we have \( \alpha^\phi \sim \beta \) if and only if \( \alpha^\phi \circ \Lambda_i \sim \beta \circ \Lambda_i \)
	for all \( i = 1, \dots, k \). Note that \( \Aut(H) \) acts on representations of \( H \) on \( A_i/pA_i \),
	and by the definition of that action we have \( \alpha^\phi \circ \Lambda_i = (\alpha \circ \Lambda_i)^\phi \), \( i = 1, \dots, k \).
	Define \( P_0 = P \) and for \( 1 \leq i \leq k \) set
	\[ P_i = \{ \phi \in P_{i-1} \mid (\alpha \circ \Lambda_i)^\phi \sim \beta \circ \Lambda_i \}. \]
	Note that \( \alpha \circ \Lambda_i \) and \( \beta \circ \Lambda_i \) are representations of \( H \) on an elementary abelian group \( A_i/pA_i \),
	hence we can compute each \( P_i \) in polynomial time by Proposition~\ref{elabcase}. Since \( P_{\alpha \to \beta} = P_k \), we are done.
\end{proof}

\section{Intertwining automorphisms}\label{secInta}

Let \( H \) and \( A \) be finite groups given by their Cayley tables, and let \( A \) be an abelian \( p \)-group
for \( p \) not dividing \( |H| \). Assume we are given representations \( \alpha, \beta : H \to \Aut(A) \).
The task of computing the set
\[ \Aut(A, \alpha \sim \beta) = \{ \psi \in \Aut(A) \mid \alpha^\psi = \beta \} \]
was called the ``intertwining automorphisms problem'' in~\cite[Problem~3]{absyl}.
Essentially this is the problem of computing all isomorphisms between two \( H \)-modules, where
\( A \) is given one \( H \)-module structure via \( \alpha \), and another \( H \)-module structure via \( \beta \).

Note that \( \Aut(A, \alpha \sim \beta) \) is either empty or a coset of \( \Aut(A, \alpha \sim \alpha) \) in \( \Aut(A) \).
Since \( A \) is given by its Cayley table, we can view \( \Aut(A) \) as a permutation group on \( A \).

In~\cite[Subsection~5.2]{absyl} this problem is solved in the case when \( A \) is elementary abelian,
and in~\cite[Section~6]{absyl} it is noted without proof that the work of Ranum~\cite{ranum} allows one to
reduce the general case to the elementary abelian one.
Here we take a different approach and prove that the intertwining automorphisms problem can be solved in polynomial time
even when \( |H| \) and \( |A| \) are not coprime.
The key tools that we need are polynomial-time algorithms for the module isomorphism problem and for the computation of the unit group of a finite ring.

\begin{proposition}[{\cite[Theorem~1.1]{iuliana} and \cite[Proposition~3.1.7]{iuliana-thesis}}]\label{modiso}
	Let \( R \) be a finite unitary ring given by cyclic generators,
	and let \( A \) and \( A_0 \) be two finite \( R \)-modules given by cyclic generators.
	There exists a polynomial-time algorithm which decides whether \( A \) and \( A_0 \) are isomorphic as \( R \)-modules,
	and if they are, one can compute an explicit \( R \)-isomorphism \( \mu : A \to A_0 \) in polynomial time.
\end{proposition}

\begin{proposition}[{\cite[Theorem~1.6]{skres}}]\label{units}
	Let \( R \) be a finite unitary ring given by its cyclic generators.
	If \( p_{\max} \) is the largest prime divisor of \( |R| \),
	then we can compute the multiplicative generators of \( R^\times \) in time polynomial in \( \log |R| \) and \( p_{\max} \).
\end{proposition}

Now we are ready to solve the intertwining automorphisms problem. We note that a very similar argument was
used in the proof of~\cite[Proposition~5.1]{skres}.

\begin{proposition}\label{pbot}
	Let \( H \) and \( A \) be groups given by their Cayley tables, and assume that \( A \) is abelian.
	If \( \alpha, \beta : H \to \Aut(A) \) are representations of \( H \) on \( A \), then one can compute
	\( \Aut(A,\alpha \sim \beta) = \{ \psi \in \Aut(A) \mid \alpha^\psi = \beta \} \) in polynomial time.
\end{proposition}
\begin{proof}
	Let \( e \) be the exponent of \( A \). Let \( R = (\mathbb{Z}/e\mathbb{Z})[H] \) be the group ring of \( H \) with
	coefficients in \( \mathbb{Z}/e\mathbb{Z} \). Note that \( R \) is a unitary ring and we can compute its cyclic generators
	in polynomial time. Representations \( \alpha \) and \( \beta \) define two \( R \)-module structures on \( A \);
	call them \( A_\alpha \) and \( A_\beta \).
	By Proposition~\ref{modiso}, we can decide in polynomial time whether \( R \)-modules \( A_\alpha \) and \( A_\beta \) are isomorphic and find an
	explicit isomorphism if they are. If \( A_\alpha \) and \( A_\beta \) are not isomorphic, then \( \Aut(A, \alpha \sim \beta) \)
	is empty and we are done. If they are isomorphic, then we can compute some isomorphism \( \mu : A \to A \), which satisfies
	\( \mu(a^{\alpha(h)}) = \mu(a)^{\beta(h)} \) for all \( a \in A \), \( h \in H \).
	This property can be rewritten as \( \mu(\mu^{-1}(a)^{\alpha(h)}) = a^{\beta(h)} \) for all \( a \in A \), \( h \in H \),
	which is equivalent to \( \alpha^\mu = \beta \). So \( \mu \in \Aut(A, \alpha \sim \beta) \) and we have
	\( \Aut(A, \alpha \sim \beta) = \Aut(A, \alpha \sim \alpha) \cdot \mu \). It suffices to compute \( \Aut(A, \alpha \sim \alpha) \).

	Note that cyclic generators of the ring \( \End(A) \) can be computed in polynomial time.
	We define a subring \( K \leq \End(A) \) by
	\[ K = \{ \psi \in \End(A) \mid \alpha(h) \cdot \psi = \psi \cdot \alpha(h) \text{ for all } h \in H \}. \]
	Since \( K \) is defined by polynomially many Diophantine equations inside of \( \End(A) \), its cyclic generators
	can be computed in polynomial time, see e.g.~\cite{dio1, dio2}.

	Let \( A_i \) be the Sylow \( p_i \)-subgroup of \( A \), so \( A = A_1 \oplus \dots \oplus A_k \).
	Then \( \End(A) = \End(A_1) \oplus \dots \oplus \End(A_k) \), and since \( \End(A_i) \) is a \( p_i \)-group
	by Proposition~\ref{autab}, the largest prime divisor of \( |\End(A)| \) is \( \max \{ p_1, \dots, p_k \} \).
	Hence the largest prime divisor of \( |K| \) is bounded by \( |A| \), and by Proposition~\ref{units}, we can compute the unit group
	\[ K^\times = \{ \psi \in \Aut(A) \mid \alpha(h) \cdot \psi = \psi \cdot \alpha(h) \text{ for all } h \in H \} \]
	in polynomial time. It is clear that \( \Aut(A, \alpha \sim \alpha) = K^\times \), so we are done.
\end{proof}

\section{Proof of Theorem~\ref{main}}\label{secMain}

Let \( A \) be a characteristic subgroup of \( G \). We will say that \( A \) is \emph{characteristically complemented}
by \( H \leq G \), if \( A \cap H = 1 \) and \( AH = G \), and moreover, for any \( \phi \in \Aut(G) \) the subgroup \( H^\phi \)
is conjugate to \( H \) in \( G \). Clearly \( H^\phi \) is also a complement to \( A \) with required properties,
so \( H \) is defined up to automorphisms of \( G \).

The following lemma shows how to reduce the problem of computing \( \Aut(G) \) to the problem of computing \( \Aut(H) \).

\begin{lemma}\label{lred}
	Let \( G \) be an A-group given by its Cayley table, and assume we are given a characteristic
	abelian \( p \)-subgroup \( A \) of~\( G \). Suppose that \( A \) is characteristically complemented by \( H \) in \( G \),
	and we are given \( H \) and permutation generators of \( \Aut(H) \).
	Then we can compute permutation generators of \( \Aut(G) \) in polynomial time.
\end{lemma}
\begin{proof}
	Let \( \phi \in \Aut(G) \) be arbitrary. Since \( A \) is characteristically complemented by \( H \), there exists
	\( x \in G \) such that \( H^\phi = H^{x^{-1}} \). In particular, for an inner automorphism \( \iota_x(g) = x^{-1}gx \), \( g \in G \),
	the composition \( \phi \circ \iota_x \) stabilizes \( H \). Define \( S = \{ \phi \in \Aut(G) \mid H^\phi = H \} \)
	and observe that
	\[ \Aut(G) = \bigcup_{x \in G} S \cdot \iota_x = \langle S \cup \{ \iota_x \mid x \in G \} \rangle, \]
	so it suffices to compute the generators of \( S \).

	An element \( \phi \in S \) is completely determined by restrictions \( \phi|_A \in \Aut(A) \)
	and \( \phi|_H \in \Aut(H) \), so there is an embedding of \( S \) into \( \Aut(A) \times \Aut(H) \).
	By identifying \( S \) with its image under this embedding, we may assume that \( S \) is a subgroup of \( \Aut(A) \times \Aut(H) \).
	We will obtain generators of \( S \) by studying its projections \( S_A \) and \( S_H \) to \( \Aut(A) \) and \( \Aut(H) \), respectively.

	Define a representation \( \alpha : H \to \Aut(A) \) of \( H \) on \( A \) by
	\( a^{\alpha(h)} = a^h \), where \( a \in A \), \( h \in H \).
	Let \( h \in H \) be an arbitrary \( p \)-element. Since \( A \) is normal in \( G \),
	the group \( \langle A, h \rangle \) is a \( p \)-subgroup of \( G \). It is abelian, as \( G \) is an A-group,
	so \( h \) acts trivially on \( A \). Let \( K \) denote the subgroup of \( H \) generated by all \( p \)-elements.
	It follows that \( K \) lies in the kernel of \( \alpha \), and we have a representation of \( \overline{H} = H/K \)
	on \( A \) defined by \( \overline{\alpha}(hK) = \alpha(h) \), \( h \in H \).
	Note that \( K \), \( \overline{H} \), \( \alpha \) and \( \overline{\alpha} \) are computable in polynomial time.

	Since \( K \) is a characteristic subgroup of \( H \), any \( \eta \in \Aut(H) \) induces an automorphism
	\( \overline{\eta} \in \Aut(\overline{H}) \). The map sending \( \eta \in \Aut(H) \) to \( \overline{\eta} \in \Aut(\overline{H}) \)
	is a homomorphism from a permutation group on \( H \) to a permutation group on \( \overline{H} \), and
	we can compute its kernel \( C \) in polynomial time. Let \( P \) denote the image of this homomorphism, i.e.
	\[ P = \{ \overline{\eta} \in \Aut(\overline{H}) \mid \eta \in \Aut(H) \}, \]
	and observe that its generators can be computed in polynomial time as images of generators of \( \Aut(H) \).

	For any \( \eta \in \Aut(H) \), the subgroup \( K \) lies in the kernel of \( \alpha^\eta \),
	so \( \overline{\alpha^\eta} \) can be naturally defined. Clearly \( \overline{\alpha}^{\overline{\eta}} = \overline{\alpha^\eta} \),
	and \( \alpha \sim \alpha^\eta \) if and only if \( \overline{\alpha} \sim \overline{\alpha}^{\overline{\eta}} \).

	Let \( \nu \in \Aut(A) \), \( \eta \in \Aut(H) \), and note that the map \( \phi : G \to G \) defined by
	\( \phi(ah) = \nu(a)\eta(h) \), \( a \in A \), \( h \in H \), lies in \( S \) if and only if
	for all \( a \in A \), \( h \in H \) we have \( \phi(a^h) = \phi(a)^{\phi(h)} \).
	By the construction of \( \phi \) and definition of \( \alpha \) this is equivalent to
	\( \nu(a^{\alpha(h)}) = \nu(a)^{\alpha(\eta(h))} \), \( a \in A \), \( h \in H \).
	After applying \( \nu^{-1} \) to both sides we yield \( \alpha(h) = \nu \cdot \alpha(\eta(h)) \cdot \nu^{-1} \), \( h \in H \).
	It follows that \( \eta \) lies in the projection of \( S \) to \( \Aut(H) \) if and only if \( \alpha \sim \alpha^{\eta^{-1}} \), in other words
	\[ S_H = \{ \eta \in \Aut(H) \mid \alpha \sim \alpha^{\eta^{-1}} \}. \]
	By the previous paragraph, \( \alpha \sim \alpha^{\eta^{-1}} \) is equivalent to \( \overline{\alpha} \sim \overline{\alpha}^{\overline{\eta}^{-1}} \), so
	\[ S_H = \{ \eta \in \Aut(H) \mid \overline{\alpha} \sim \overline{\alpha}^{\overline{\eta}^{-1}} \}. \]

	Note that \( K \) contains all Sylow \( p \)-subgroups of \( H \), so \( |\overline{H}| \) is not divisible by~\( p \).
	We can now apply Proposition~\ref{ptop} to \( A \), \( \overline{H} \) and \( \overline{\alpha} \)
	and compute
	\[ P_{\overline{\alpha} \to \overline{\alpha}} = \{ \overline{\eta} \in P \mid
	\overline{\alpha} \sim \overline{\alpha}^{\overline{\eta}^{-1}} \} \]
	in polynomial time.
	The group \( S_H \) is the full preimage of \( P_{\overline{\alpha} \to \overline{\alpha}} \)
	under the map \( \overline{\phantom{a}} : \Aut(H) \to \Aut(\overline{H}) \). Therefore we can compute the generators of \( S_H \)
	in polynomial time by computing some preimages of generators of \( P_{\overline{\alpha} \to \overline{\alpha}} \) and adding the generators of~\( C \)
	which were computed earlier. Let \( \eta_1, \dots, \eta_t \) denote the generators of \( S_H \).

	For \( \eta \in S_H \) define \( S_A(\eta) = \{ \nu \in \Aut(A) \mid (\nu, \eta) \in S \} \).
	Note that \( S_A(\eta) \) is a coset of \( S_A(1_H) \), where \( 1_H \) denotes the identity automorphism of \( H \).
	For every \( \eta \in S_H \) choose some \( \nu_\eta \in \Aut(A) \) such that \( (\nu_\eta, \eta) \in S \). Then
	\[ S = \bigcup_{\eta \in S_H} S_A(\eta) \cdot (1_A, \eta) = \bigcup_{\eta \in S_H} S_A(1_H) \cdot (\nu_\eta, \eta) =
	\langle S_A(1_H) \cup \{ (\nu_\eta, \eta) \mid \eta \in S_H \} \rangle. \]
	We claim that \( S = \langle S_A(1_H) \cup \{ (\nu_{\eta_i}, \eta_i) \mid i = 1, \dots, t \} \rangle \).
	It is clear that the right hand side lies in \( S \), so it suffices to show that \( \langle (\nu_{\eta_i}, \eta_i) \mid i = 1, \dots, t \rangle \)
	contains a complete set of representatives for cosets of \( S_A(1_H) \) in \( S \).
	Given \( \eta \in S_H \), we can express it as \( \eta = \eta_{i_1} \cdots \eta_{i_k} \) for some \( i_1, \dots, i_k \in \{ 1, \dots, t \} \).
	Then \( (\nu_{\eta_{i_1}}, \eta_{i_1}) \cdots (\nu_{\eta_{i_k}}, \eta_{i_k}) = (\nu', \eta) \in S \) and hence
	\( S_A(\eta) \cdot (1_A, \eta) = S_A(1_H) \cdot (\nu', \eta) \), as required.

	Now the generators of \( S_A(1_H) = \Aut(A, \alpha \sim \alpha) \) can be computed in polynomial time by Proposition~\ref{pbot}.
	For every \( \eta_i \), \( i = 1, \dots, t \), we can find \( \nu_{\eta_i} \in \Aut(A) \) such that \( (\nu_{\eta_i}, \eta_i) \in S \)
	in polynomial time by Proposition~\ref{pbot} applied to \( \alpha \) and~\( \alpha^{\eta_i^{-1}} \).
	Hence we can compute the generators of \( S \) and thus of \( \Aut(G) \) in polynomial time.
\end{proof}

Although we will apply the following lemma to solvable A-groups only, we formulate it in the general case
in view of potential future applications.

\begin{lemma}\label{sgen}
	Let \( G \) be an A-group with a nonabelian solvable radical.
	Then there exists a nontrivial characteristic abelian \( p \)-subgroup \( A \)
	such that it is characteristically complemented by some \( H \leq G \).
	If \( G \) is given by its Cayley table, then one can also compute \( A \) and \( H \) in polynomial time.
\end{lemma}
\begin{proof}
	Note that if we found (in polynomial time) a characteristic abelian subgroup \( A \) which is characteristically complemented by \( H \),
	then we can find an abelian \( p \)-subgroup with the required properties.
	Indeed, let \( A_p \) be the Sylow \( p \)-subgroup of \( A \) for some prime \( p \) dividing \( |A| \).
	Let \( A_{p'} \) be the product of all the other Sylow subgroups of \( A \), so \( A = A_p A_{p'} \) and \( A_p \cap A_{p'} = 1 \).
	Clearly \( A_p \) is a characteristic abelian \( p \)-subgroup of \( G \), and it is characteristically complemented by \( A_{p'}H \).
	Since Sylow subgroups can be computed in polynomial time, \( A_p \) and \( A_{p'}H \) are also computable in polynomial time.

	Let \( S \) be the solvable radical of \( G \); recall that it can be computed in polynomial time.
	Let \( N \) be the penultimate member of the derived series of \( S \).
	Since \( S \) is nonabelian by our assumptions, \( N \) is nonabelian and \( A = [N, N] \) is abelian.
	Both \( N \) and \( A \) can be computed in polynomial time as the derived series is computable in polynomial time.
	Let \( p_1, \dots, p_k \) be the set of prime divisors of \( |N| \). Since \( N \) is solvable,
	for every \( i = 1, \dots, k \) there exists a Hall \( p_i' \)-subgroup \( N_i \leq N \).
	We can compute \( N_1, \dots, N_k \) in polynomial time by~\cite[Theorem~5.3]{kantor}.
	Set \( P_i = \bigcap_{j \neq i} N_j \) for \( i = 1, \dots, k \). Then \( P_i \) is a Sylow \( p_i \)-subgroup of \( N \)
	and \( P_1, \dots, P_k \) form a Sylow system, that is, \( P_iP_j = P_jP_i \) for all \( i, j \).

	Let \( H = \{ g \in G \mid P_i^g = P_i,\, i = 1, \dots, k \} \) be the relative system normalizer of \( P_1, \dots, P_k \) in \( G \).
	By~\cite[Corollary~4.6]{broshi}, we have \( A \cap H = 1 \) and \( AH = G \). For any \( \phi \in \Aut(G) \) the subgroup \( H^\phi \)
	is a relative system normalizer of \( P_1^\phi, \dots, P_k^\phi \). In a solvable group all Sylow systems are conjugate, so
	\( H^\phi \) is conjugate to \( H \) by an element from \( N \), and hence \( A \) is characteristically complemented by \( H \).
\end{proof}

We are finally ready to give the proof of the main theorem.
\medskip

\noindent
\emph{Proof of Theorem~\ref{main}.}
By Theorem~\ref{treduct}, it suffices to solve \( \mathrm{GrpAGEN}_\X \) for the class \( \X \) of solvable A-groups.
We reason by induction on the order of the input group \( G \).

If \( G \) is abelian, we can compute \( \Aut(G) \) by~\cite{absyl} or by applying Proposition~\ref{units} to \( \End(G) \).
If \( G \) is nonabelian, then by Lemma~\ref{sgen} there exists a nontrivial
characteristic abelian \( p \)-subgroup \( A \) of \( G \) which is characteristically complemented by \( H \leq G \).
Moreover, \( A \) and \( H \) can be found in polynomial time. Since \( |H| < |G| \) and \( H \) is a solvable A-group, we can compute \( \Aut(H) \)
in polynomial time by a recursive call to \( \mathrm{GrpAGEN}_\X \). Now we can compute \( \Aut(G) \) in polynomial time by Lemma~\ref{lred}.

It is left to note that the overall algorithm works in polynomial time since all steps are performed in polynomial time
and the reduction from \( G \) to a smaller group \( H \) involves no branching and always reduces the size of the input.
\qed

\section{Reduction to the group automorphism problem}\label{secAut}

We will consider a class of finite groups \( \X \) which is closed with respect to
taking direct factors and direct products. In other words, if \( G \in \X \)
and \( G \simeq A \times B \) then \( A, B \in \X \), and if \( A, B \in \X \) then \( A \times B \in \X \).
Note that the classes of all finite groups, solvable, nilpotent groups,
groups with trivial solvable radical, and, most importantly for our purposes, A-groups satisfy this definition.

We define the following set of computational problems, where the subscript \( \X \) means that input groups lie in the class \( \X \).
Recall that input groups are given by their Cayley tables and the group of automorphisms of a finite group is given by its permutation generators.
The automorphic partition of a group \( G \) is the partition of \( G \) into \( \Aut(G) \)-orbits.
\begin{align*}
	&\mathrm{GrpISO}_\X(G, H): \text{ determine if two groups } G \text{ and } H \text{ are isomorphic or not.}\\
	&\mathrm{GrpIMAP}_\X(G, H): \text{ find an isomorphism between } G \text{ and } H \text{ if it exists.}\\
	&\mathrm{GrpICOUNT}_\X(G, H): \text{ count the number of isomorphisms between } G \text{ and } H.\\
	&\mathrm{GrpACOUNT}_\X(G): \text{ count the number of automorphisms of } G.\\
	&\mathrm{GrpAGEN}_\X(G): \text{ find generators of the full automorphism group of } G.\\
	&\mathrm{GrpAPART}_\X(G): \text{ find the automorphic partition of } G.
\end{align*}

For two computational problems \( P_1 \) and \( P_2 \) we write \( P_1 \propto P_2 \) if \( P_1 \) is polynomial-time reducible to \( P_2 \).
Problems \( P_1 \) and \( P_2 \) are polynomial-time equivalent if \( P_1 \propto P_2 \) and \( P_2 \propto P_1 \).

\begin{theorem}\label{treduct}
	Let \( \X \) be a class of finite groups which is closed with respect to taking direct factors and direct products.
	The following holds:
	\begin{enumerate}[\rm(i)]
		\item \( \mathrm{GrpISO}_\X \propto \mathrm{GrpACOUNT}_\X \propto \mathrm{GrpAGEN}_\X \),
		\item \( \mathrm{GrpISO}_\X \propto \mathrm{GrpAPART}_\X \propto \mathrm{GrpAGEN}_\X \),
		\item \( \mathrm{GrpIMAP}_\X \propto \mathrm{GrpAGEN}_\X \),
		\item \( \mathrm{GrpICOUNT}_\X \) and \( \mathrm{GrpACOUNT}_\X \) are polynomial-time equivalent.
	\end{enumerate}
\end{theorem}

The proof of Theorem~\ref{treduct} essentially repeats the proof in~\cite{grpreduct} with the main difference
that we note that our intermediate constructions lie in the class \( \X \). We provide the full proof here for the reader's convenience.

We need two preliminary results from~\cite{grpreduct}.

\begin{lemma}[{\cite[Lemma~3.1]{grpreduct}}]\label{lcounthom}
	Let \( A \) and \( B \) be groups given by their Cayley tables, and assume that \( B \) is abelian.
	Then one can compute \( |\Hom(A, B)| \) in polynomial time.
\end{lemma}

\begin{lemma}[{\cite[Lemma~3.2]{grpreduct}}]\label{linvar}
	Let \( G \) and \( H \) be finite directly indecomposable nonabelian groups, and let \( S = G \times Z(H) \).
	Then \( S \leq G \times H \) is invariant under \( \Aut(G \times H) \) if and only if \( G \) is not isomorphic to~\( H \).
\end{lemma}

\noindent
\emph{Proof of Theorem~\ref{treduct}.}
The reductions \( \mathrm{GrpACOUNT}_\X \propto \mathrm{GrpAGEN}_\X \) and \( \mathrm{GrpAPART}_\X \propto \mathrm{GrpAGEN}_\X \)
follow from standard permutation group algorithms for finding the order and orbits of a permutation group.

Note that we can assume that input groups \( G, H \) in problems \( \mathrm{GrpISO}_\X \) and \( \mathrm{GrpIMAP}_\X \) are directly indecomposable.
Indeed, by~\cite{kayal} in polynomial time we can decompose our groups as \( G \simeq G_1 \times \dots \times G_k \) and \( H \simeq H_1 \times \dots \times H_m \),
where \( G_i \), \( i = 1, \dots, k \), and \( H_j \), \( j = 1, \dots, m \), are directly indecomposable.
By the definition of the class \( \X \), we have \( G_i, H_j \in \X \). By the Krull--Schmidt theorem, these decompositions of \( G \) and \( H \)
are unique up to reordering of the factors, so it suffices to check isomorphism (or find an explicit isomorphism in the case of \( \mathrm{GrpIMAP}_\X \))
for \( k \cdot m \leq \log |G| \cdot \log |H| \) pairs of directly indecomposable groups from~\( \X \).
Additionally we can assume that our input groups are nonabelian,
since in the abelian case \( \mathrm{GrpISO}_\X \) and \( \mathrm{GrpIMAP}_\X \) are solvable in polynomial time.

Given groups \( A, B, C \) and maps \( \theta : B \to C \), \( \phi : A \to B \), we define their product as
\( (\theta * \phi) : A \to C \) where \( (\theta * \phi)(a) = \theta(\phi(a)) \), \( a \in A \).
The product of two homomorphisms is a homomorphism.
Given maps \( \theta, \phi : A \to B \), we define their sum as \( (\theta + \phi) : A \to B \) where \( (\theta + \phi)(a) = \theta(a)\phi(a) \), \( a \in A \).
If \( \theta \) and \( \phi \) are homomorphisms and \( \theta(A) \) commutes elementwise with \( \phi(A) \), then \( \theta + \phi \) is also a homomorphism.

Let \( G \) and \( H \) be directly indecomposable nonabelian groups from the class~\( \X \).
The following description of \( \Aut(G \times H) \) was obtained in~\cite{bidwell, bidwell2}.
Define a group of formal matrices
\[ \mathcal{A} = \left\{ \begin{pmatrix} \alpha & \beta \\ \gamma & \delta \end{pmatrix} \mid \begin{aligned} &\alpha \in \Aut(G), \beta \in \Hom(H, Z(G)),\\
&\gamma \in \Hom(G, Z(H)), \delta \in \Aut(H) \end{aligned} \right\}, \]
where the group operation is matrix multiplication, and the products and sums of maps are computed as specified above.
A matrix from \( \mathcal{A} \) acts on an element \( (g, h) \in G \times H \) by multiplying the column-vector \( (g, h)^T \) by the matrix from the left:
\[ \tag{\(\star\)} \begin{pmatrix} \alpha & \beta \\ \gamma & \delta \end{pmatrix} (g, h)^T = (\alpha(g)\beta(h), \gamma(g)\delta(h))^T. \]
It is shown in~\cite{bidwell, bidwell2} that this gives a faithful action of \( \mathcal{A} \) on \( G \times H \) by automorphisms.
Moreover, if \( G \) and \( H \) are not isomorphic, then \( \Aut(G \times H) \simeq \mathcal{A} \), and if \( G \) and \( H \) are isomorphic, then
\( \Aut(G \times H) \simeq \mathcal{A} \rtimes \langle \phi \rangle \),
where \( \phi \) is an automorphism of order \( 2 \) which swaps the direct factors of \( G \times H \).

To prove the reduction \( \mathrm{GrpISO}_\X \propto \mathrm{GrpACOUNT}_\X \), note that
\begin{multline*}
	|\Aut(G \times H)| = |\mathcal{A}| \cdot \epsilon_{G,H} = \\ = |\Aut(G)| \cdot |\Aut(H)| \cdot |\Hom(G, Z(H))| \cdot |\Hom(H, Z(G))| \cdot \epsilon_{G,H},
\end{multline*}
where \( \epsilon_{G,H} = 1 \) if \( G \not\simeq H \) and \( \epsilon_{G,H} = 2 \) otherwise. By definition of the class \( \X \) and our assumptions,
\( G \times H \), \( G \) and \( H \) lie in \( \X \), so we can use three calls to the \( \mathrm{GrpACOUNT}_\X \) oracle to compute
\( |\Aut(G \times H)| \), \( |\Aut(G)| \) and \( |\Aut(H)| \). By Lemma~\ref{lcounthom}, we can compute \( |\Hom(G, Z(H))| \) and \( |\Hom(H, Z(G))| \)
in polynomial time, and so we can compute \( \epsilon_{G,H} \). Part~(i) is proved.

In order to reduce \( \mathrm{GrpISO}_\X \) to \( \mathrm{GrpAPART}_\X \), note that \( G \times H \in \X \) and we can compute the automorphic
partition of \( G \times H \) by one call to the \( \mathrm{GrpAPART}_\X \) oracle. By Lemma~\ref{linvar}, groups \( G \) and \( H \) are not isomorphic if and only if
\( G \times Z(H) \) is a union of \( \Aut(G \times H) \)-orbits, which can be tested in polynomial time. Part~(ii) is proved.

We prove part~(iii). Set \( S = G \times Z(H) \). Since \( G \times H \in \X \), we can compute \( \Aut(G \times H) \) by a call to the \( \mathrm{GrpAGEN}_\X \) oracle.
If all generators of \( \Aut(G \times H) \) stabilize \( S \), then by Lemma~\ref{linvar}, \( G \) and \( H \) are not isomorphic and we are done.
Suppose \( \psi \in \Aut(G \times H) \) is a generator for which \( \psi(S) \neq S \), so there exists an isomorphism \( \pi : G \to H \).
For \( g \in G \) define \( \chi(g) \in H \) as the projection of \( \psi((g, 1)) \in G \times H \) to \( H \).
The map \( \chi : G \to H \) is a well-defined homomorphism which can be computed in polynomial time.

Since \( G \simeq H \), we have \( \Aut(G \times H) \simeq \mathcal{A} \rtimes \langle \phi \rangle \),
where \( \phi((g, h)) = (\pi^{-1}(h), \pi(g)) \), \( g \in G \), \( h \in H \). Hence \( \psi(x) = \phi(\theta(x)) \)
for some \( \theta \in \mathcal{A} \). In the notation of formula~\( (\star) \) we have
\( \theta((g, 1)) = (\alpha(g)\beta(1), \gamma(g)\delta(1)) = (\alpha(g), \gamma(g)) \), hence
\[ \psi((g, 1)) = \phi(\theta((g,1))) = \phi((\alpha(g), \gamma(g))) = (\pi^{-1}(\gamma(g)), \pi(\alpha(g))). \]
We have \( \chi(g) = \pi(\alpha(g)) \), and since \( \alpha \in \Aut(G) \), the map \( \chi : G \to H \) is an isomorphism.
Part~(iii) is proved.

To prove part~(iv), note that \( \mathrm{GrpACOUNT}_\X \propto \mathrm{GrpICOUNT}_\X \) holds trivially,
as the size of the full automorphism group is equal to the number of isomorphisms between the group and itself.
To show \( \mathrm{GrpICOUNT}_\X \propto \mathrm{GrpACOUNT}_\X \), note that since \( \mathrm{GrpISO}_\X \propto \mathrm{GrpACOUNT}_\X \),
we can check whether \( G \) and \( H \) are isomorphic or not. If they are isomorphic, then the number of isomorphisms
is equal to \( |\Aut(G)| \), which can be computed with one call to the \( \mathrm{GrpACOUNT}_\X \) oracle. If the groups are not isomorphic,
then we output \( 0 \). Part~(iv) is proved. \qed

\section{Final remarks}\label{secFinal}

Although Theorem~\ref{main} applies to solvable A-groups only, the proof is structured in a way that the solvability assumption
is used only at the very end. The reason for such organization of the proof is that in the earlier version of this text the author erroneously
claimed a polynomial-time isomorphism test for arbitrary, not necessarily solvable, A-groups. The mistake in the old argument was after Lemma~\ref{sgen}
and was brought to the author's attention by J.A.~Grochow; it was stated that an A-group \( G \) with a nontrivial abelian solvable radical has
a nontrivial characteristic abelian \( p \)-subgroup which is characteristically complemented
in~\( G \). This is not true in general, for instance, let \( G \) be the semidirect product of \( \mathrm{Alt}(5)^7 \) by \( C_{49} \),
where \( \mathrm{Alt}(5)^7 \) is a direct product of \( 7 \) copies of \( \mathrm{Alt}(5) \), and \( C_{49} \) is the cyclic group of order \( 49 \)
which acts on \( \mathrm{Alt}(5)^7 \) by transitively permuting the direct factors.
One can easily check that \( G \) is an A-group, and its solvable radical of order \( 7 \) is not characteristically complemented in~\( G \).

Note that Lemmata~\ref{lred} and~\ref{sgen} reduce the isomorphism problem for arbitrary A-groups to the problem of computing
the full automorphism group of an A-group with the abelian solvable radical. Deciding isomorphism of groups with abelian solvable radicals
is an important open problem, see~\cite[Open Problem~8.1]{grochow}. The most general known results in this direction are due to
Babai, Codenotti, Qiao~\cite{trivrad} for groups with trivial solvable radical, and Grochow, Qiao~\cite{grochow} for groups with the elementary abelian
solvable radical and certain restrictions on the number of minimal normal subgroups in the quotient by the solvable radical.

It is natural to hope that structural properties of A-groups should help with this problem. For example, the main result of Broshi, see~\cite[Theorem~5.1]{broshi},
implies that an A-group \( G \) splits over the last term of its derived series \( G^{(\infty)} \), while \( G^{(\infty)} \) splits over its solvable radical.
In particular, the main bottleneck for generalizing Theorem~\ref{main} to nonsolvable A-groups seems to be the case when the solvable radical of \( G \)
is abelian and the solvable radical of \( G^{(\infty)} \) is trivial. It is possible that cohomological techniques in the spirit of~\cite{grochow}
or a generalization of the approach in~\cite{trivrad} could be used in this case, and the author plans to return to this question in the future.

\section{Acknowledgements}

The author is indebted to J.A.~Grochow for pointing out a mistake in the old version of this text, and for many useful remarks.
The author also expresses his gratitude to I.N.~Ponomarenko and A.V.~Vasil'ev for helpful suggestions that improved the text.

The research was carried out within the framework of the Sobolev Institute of Mathematics state contract (project FWNF-2026-0017).

\bigskip

\noindent
\emph{Saveliy V. Skresanov}

\noindent
\emph{Sobolev Institute of Mathematics, 4 Acad. Koptyug avenue,\\ 630090 Novosibirsk, Russia}

\noindent
\emph{Email address: skresan@math.nsc.ru}

\end{document}